\documentclass[12pt,a4paper]{scrartcl}
\usepackage[utf8]{inputenc}
\usepackage[T1]{fontenc}
\usepackage[english]{babel}

\usepackage{graphicx}
\usepackage{latexsym}
\usepackage{amsmath,amssymb,amsthm}

\usepackage{enumitem}

\usepackage{lineno}

\usepackage{bbm}

\usepackage{mathtools}
\usepackage{abstract}

\theoremstyle{definition}

\newtheorem{definition}{Definition}[section]
\newtheorem{fact}[definition]{Fact}
\newtheorem{theorem}[definition]{Theorem}
\newtheorem{lemma}[definition]{Lemma} 
\newtheorem{corollary}[definition] {Corollary}                

\numberwithin{equation}{section} 

\title{The cofinality of the strong measure zero ideal for $\kappa$ inaccessible}
\author{Johannes Philipp Sch\"urz \footnote{supported by FWF project I3081} }
\date{} 

\begin{document}

\maketitle

\begin{abstract}
We investigate the cofinality of the strong measure zero ideal for $\kappa$ inaccessible, and show that it is independent of the size of $2^\kappa$.
\end{abstract} 

\section{Introduction}

In this paper we continue to investigate the strong measure zero sets on $2^\kappa$ for $\kappa$ at least inaccessible (see \cite{Me}).\\
\\
Halko generalized the notion of strong measure zero to uncountable cardinals as follows (see \cite{Halko}):

\begin{definition}
Let $X \subset 2^\kappa$. We say $X$ has strong measure zero iff
$$\forall f\in \kappa^\kappa \, \exists (\eta_i)_{i < \kappa} \colon \big (\, \forall i < \kappa \,\, \eta_i \in 2^{f(i)} \, \big ) \land  X \subset \bigcup_{i< \kappa} [\eta_i].$$
We shall denote this by $X \in \mathcal{SN}$.
\end{definition}

The following is an easy fact:

\begin{fact}
$\mathcal{SN}$ is a $\kappa$-closed, proper ideal on $2^\kappa$ which contains all singletons.
\end{fact}

Therefore, the following cardinal characteristics are well defined:

\begin{definition}
$$\text{add}(\mathcal{SN}):= \min \{\vert \mathcal{F} \vert \colon \mathcal{F} \subset \mathcal{SN} \land \bigcup \mathcal{F} \notin \mathcal{SN}\}$$
$$\text{cov}(\mathcal{SN}):= \min \{\vert \mathcal{F} \vert \colon \mathcal{F} \subset \mathcal{SN} \land \bigcup \mathcal{F} =2^\kappa\}$$
$$\text{non}(\mathcal{SN}):= \min \{\vert X \vert \colon X \subset 2^\kappa \land  X \notin \mathcal{SN}\}$$
$$\text{cof}(\mathcal{SN}):= \min \{\vert \mathcal{F} \vert \colon \mathcal{F} \subset \mathcal{SN} \land 
\forall X \in \mathcal{SN} \, \exists Y \in \mathcal{F} \,\, X \subset Y\}$$
\end{definition}

In \cite{Yorioka}, Yorioka introduced the so-called Yorioka ideals approximating the ideal of strong measure zero sets on $2^\omega$. We will generalize this notion to $\kappa$ and use it to investigate $\text{cof}(\mathcal{SN})$. Eventually, we shall show that $\text{cof}(\mathcal{SN}) < \mathfrak{c}_\kappa$, $\text{cof}(\mathcal{SN}) = \mathfrak{c}_\kappa$ as well as $\text{cof}(\mathcal{SN}) > \mathfrak{c}_\kappa$ are all consistent relative to ZFC.\\
\\
In the last chapter we first generalize the Galvin–Mycielski–Solovay theorem (see \cite{Judah}) to $\kappa$ inaccessible. This result was originally proven by Wohofsky (see \cite{Wohofsky}). Then we follow Pawlikowski \cite{Pawlikowski} and show the relative consistency of $\text{cov}(\mathcal{SN}) < \text{add}(\mathcal{M}_\kappa)$ for $\kappa$ strongly unfoldable (see Definition \ref{D3}).\\
\\
Strong measure zero sets for $\kappa$ regular uncountable have also been studied in \cite{Halko} and \cite{Halko2}. More questions have been stated in \cite{Questions}.\\
\\
Last but not least, I would like to thank my advisor Martin Goldstern for the very helpful comments during the preparation of this paper.

\section{Prerequisites}

We start with several definitions:

\begin{definition}
Let $f,g \in \kappa^\kappa$ and $f$ strictly increasing.
\begin{itemize}
\item We define the order $\ll$ on $\kappa^\kappa$ as follows:
$f\ll g$ iff $\forall \alpha < \kappa \, \, \exists \beta < \kappa \, \, \forall i \geq \beta \colon g(i) \geq f(i^\alpha)$. Here $i^\alpha$ is defined using ordinal arithmetic.
\item For $\sigma \in (2^{<\kappa})^\kappa$ define $g_\sigma$ as follows: $g_\sigma (i):= \text{dom}(\sigma(i))$.
\item For $\sigma \in (2^{<\kappa})^\kappa$ define $Y(\sigma) \subseteq 2^\kappa$ as follows: $Y(\sigma):= \bigcap_{i < \kappa} \bigcup_{j \geq i} \, [\sigma(j)]$.
\item Define $\mathcal{S}(f) \subseteq (2^{<\kappa})^\kappa$ as follows: $\mathcal{S}(f):=\{ \sigma \in (2^{<\kappa})^\kappa \colon f \ll g_\sigma\}$.
\item Define $A \subseteq 2^\kappa$ to be $f$-small iff there exists $\sigma \in \mathcal{S}(f)$ such that $A \subseteq Y(\sigma)$.
\item Set $\mathcal{I}(f):=\{A \subseteq 2^\kappa \colon A \, \text{is} \,  f\text{-small} \}$.
\end{itemize}
\end{definition}

\begin{definition}
Let $f \in \kappa^\kappa$ be strictly increasing and let $\sigma \in \mathcal{S}(f)$. For every $\alpha < \kappa$ we define $M_\sigma^\alpha$ to be the minimal ordinal $\geq 2$ such that $\forall i \geq M_\sigma ^\alpha \colon g_\sigma(i) \geq f(i^\alpha)$.
%Define $m_k :=\sup\{M_j^{3 \cdot \alpha} \colon j, \alpha \leq k\}$.
\end{definition}

\begin{lemma}
$\mathcal{I}(f)$ forms a $\kappa$-closed ideal.
\end{lemma}

\begin{proof}
$\mathcal{I}(f)$ is obviously closed under subsets. So we must show that it is also closed under $\kappa$-unions.\\
\\
Let $(A_k)_{k< \kappa}$ be a family of $f$-small sets and $(\sigma_k)_{k<\kappa}$ such that $A_k \subseteq Y(\sigma_k)$. We shall find a $\tau \in \mathcal{S}(f)$ such that $\bigcup_{k <\kappa} Y(\sigma_k) \subseteq Y(\tau)$. For every $k < \kappa$ let $g_k:=g_{\sigma_k}$ and let $M_k^\alpha:= M_{\sigma_k}^\alpha$ for ever $\alpha < \kappa$. Define $m_k :=\sup\{M_j^{3 \cdot \alpha} \colon j, \alpha \leq k\}$.\\
\\
We need the following definitions for $i \geq m_0$:
\begin{itemize}
\item Define $c(i) >0$ such that $m_0 + \sum_{j < c(i)} j \cdot m_j \leq i < m_0 + \sum_{j \leq c(i)} j \cdot m_j$.
\item Define $d(i):= m_0 +  \sum_{j < c(i)} j \cdot m_j$.
\item Define $a(i)$ and $b(i)$ such that $i - d(i) = c(i) \cdot a(i) + b(i)$ with $b(i) < c(i)$.
\item Define $e(i):=\sum_{j < c(i)}  m_j$.
\end{itemize}
The following are immediate consequences for $i > m_0$:
\begin{itemize}
\item $i \geq m_0 + \sum_{j < c'} j \cdot m_j \Rightarrow c(i) \geq c'$
\item $a(i) < m_{c(i)}$
\item $0 \leq b(i) < c(i) \leq e(i)$.
\item $d(i) \leq c(i) \cdot e(i)$  (show by induction)
\item $i= d(i) + c(i) \cdot a(i) + b(i) < c(i) \cdot e(i) + c(i) \cdot a(i) + c(i) = c(i) \cdot (e(i) + a(i) + 1) \leq ((e(i) + a(i)) \cdot (e(i) + a(i) + 1) \leq ((e(i) + a(i))^3$
\item $\forall k <\kappa \,\, \forall^{\infty} l < \kappa \,\, \exists i < \kappa \colon e(i) + a(i)=l \land b(i)=k$
\end{itemize}
The last statement can be deduced as follows: Given $k < \kappa$ choose $l \geq \sum_{j \leq k}  m_j$, so there is a $\tilde{c}>k$ such that $\sum_{j < \tilde{c}}  m_j \leq l <\sum_{j \leq \tilde{c}}  m_j$. Define $i:= m_0+ \sum_{j < \tilde{c}} j \cdot m_j + \tilde{c} \cdot (l - \sum_{j < \tilde{c}}  m_j) + k $. Then $c(i)=\tilde{c}$. (This follows because $m_{\tilde{c}} \geq (l - \sum_{j < \tilde{c}}  m_j) +1$ and so $\tilde{c} \cdot m_{\tilde{c}} > \tilde{c} \cdot (l - \sum_{j < \tilde{c}}  m_j) +k$.) Hence $d(i) = m_0+ \sum_{j < \tilde{c}} j \cdot m_j$, $a(i)= l - \sum_{j < \tilde{c}}  m_j$, $b(i)=k$ and $e(i)= \sum_{j < \tilde{c}}  m_j$. Hence $e(i) + a(i)=l$.

\begin{figure}[h]
 \centering
 \includegraphics[scale=0.9]{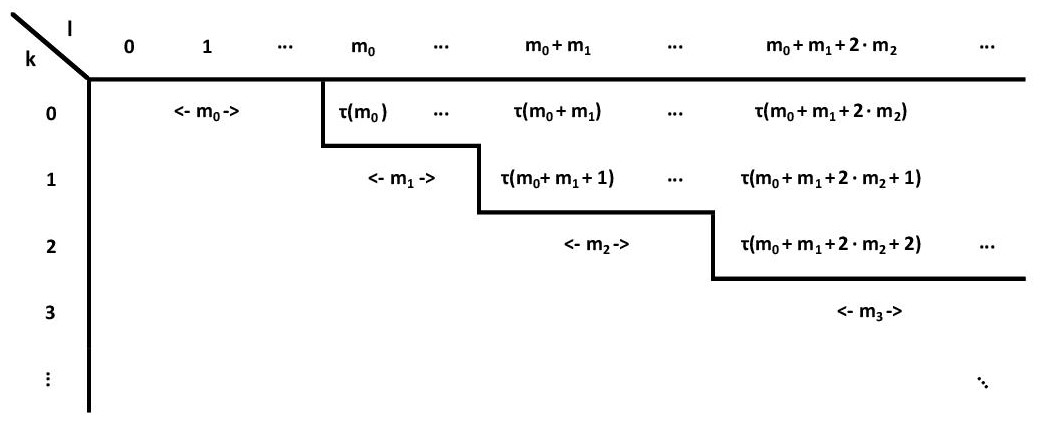}
 \caption{Definition of $\tau$}
 \label{Figure 1}
\end{figure}

Now we can define $\tau$: If $i\geq m_0$ set $\tau (i) = \sigma_{b(i)} (e(i)+ a(i))$. Else set $\tau (i)= \langle \, \rangle$. We must show that $\tau$ has the required properties.\\
First let us show that $\bigcup_{k < \kappa} Y(\sigma_k) \subseteq Y(\tau)$.  Let $x \in Y(\sigma_k)$ for some $k< \kappa$. If $l < \kappa$ is large enough, then there exists $i < \kappa$ such that $\tau (i) = \sigma_k (l)$. Hence $x \in Y(\tau)$.\\
Now we must show that $\tau \in \mathcal{S}(f)$. Let $\alpha < \kappa$ be arbitrary. If $i \geq m_0+\sum_{j \leq \alpha} j \cdot m_j$ (hence $c(i)> \alpha$), then the following (in-)equalities hold: $\tau (i) = \sigma_{b(i)}(e(i)+a(i))$ by definition, $g_{b(i)}(e(i)+a(i)) \geq f((e(i)+a(i))^{3 \cdot \alpha})$ since $e(i)+a(i) \geq M_{b(i)}^{3 \cdot \alpha}$ and $b(i), \alpha < c(i)$, and $f((e(i)+a(i))^{3 \cdot \alpha}) \geq f(i^\alpha)$ since $f$ is strictly increasing and $i \leq (e(i)+a(i))^3 $. Hence $g_{\tau} (i) \geq f(i^\alpha)$.
\end{proof}

The following fact is straightforward:

\begin{fact}
$\mathcal{SN} = \bigcap_{f \in \kappa^\kappa} \mathcal{I}(f)$.
\end{fact}

The next lemma implicitly shows that comeager sets cannot be strong measure zero:

\begin{lemma} \label{L1}
Let $A \subseteq 2^\kappa$ be comeager. Then there exists $f \in \kappa^\kappa$ such that for every $f$-small set $B$ the set $A\setminus  B $ is non-empty.
\end{lemma}

\begin{proof}
Let $A$ be comeager. We shall show that $A$ contains a perfect set, hence there exists $f \in \kappa^\kappa$ such that $A$ is not $f$-small: Let $P \subset 2^\kappa$ be a perfect set, and let $T \subset 2^{< \kappa}$ be a perfect tree such that $[T]=P$. Pick a function $f \in \kappa^\kappa$ such that $f(\alpha) > \sup \{ \text{dom}(t) \colon t \,\, \text{is in the}\,\, \alpha \text{th} \,\, \text{splitting level of} \,\,T\}$. Let $\sigma \in \mathcal{F}(f)$ be arbitrary. Then an $x \in [T] \setminus Y(\sigma)$ can be constructed by induction.\\
\\
Now assume that $A= \bigcap_{i<\kappa}D_{i+1}$ where $D_{i+1}$ are open dense and decreasing. We inductively construct a perfect tree $T \subset 2^{<\kappa}$ such that no branches die out and $[T] \subset A$:
\begin{itemize}
\item Set $T_0:=\{t_{\langle \, \rangle}\}$ where $t_{\langle \, \rangle}:= \langle \,\rangle$.
\item If $i=i'+1$ assume inductively that $T_{i'}=\{t_\eta \colon \eta \in 2^{i'}\}$ has already been defined and for every $t_\eta \in T_{i'}$ it holds that $[t_\eta] \subset \bigcap_{j < i'} D_{j+1}$. For every $t_\eta \in T_{i'}$ find $t_\eta' \triangleright t_\eta$ such that $[t_\eta'] \subset D_i$. Set $t_{\eta^\frown \langle i \rangle}:=t_\eta'^\frown \langle i \rangle$ and $T_i:=\{t_{\eta'} \colon \eta' \in 2^i\}$.
\item If $\gamma$ is a limit and $\eta \in 2^\gamma$ define $t_\eta:= \bigcup_{\eta' \triangleleft \, \eta} t_{\eta'}$ and set $T_\gamma:=\{t_\eta \colon \eta \in 2^\gamma\}$. Then we can deduce that $[t_\eta] \subset \bigcap_{j< \gamma} D_{j+1}$ for every $t_\eta \in T_\gamma$.
\end{itemize}
It follows from the construction that (the downward closure of) $T:= \bigcup_{i < \kappa} T_i$ is a perfect tree and that $[T] \subset A$.
\end{proof}

Conversely, the following fact holds true:
\begin{fact}
For every $f \in \kappa$ strictly increasing there exists a comeager $A \subset 2^\kappa$ such that $A \in \mathcal{I}(f)$.

\end{fact}

\begin{lemma} \label{L2}
Assume $CH$, i.e. $\vert 2^\kappa \vert = \kappa^+$, and let $(f_\alpha)_{\alpha < \kappa^+}$ be a $\kappa$-scale such that $f_\alpha$ is strictly increasing. Then there exists a matrix $(A_\alpha ^\beta)_{\alpha <\kappa^+} ^{\beta < \kappa^+}$ with the following properties:
\begin{itemize}
\item $\forall \alpha, \beta <\kappa^+ \colon A_\alpha^\beta \subseteq 2^\kappa \, \text{is comeager and}  \, f_\alpha\text{-small}$.
\item $\forall \alpha, \beta, \beta' <\kappa^+ \colon \beta \leq \beta' \Rightarrow A_\alpha^\beta \subseteq A_\alpha^{\beta'}$.
\item $\forall \alpha < \kappa^+ \,\, \forall f_\alpha \text{-small} \, B \subseteq 2^\kappa \,\, \exists \beta <\kappa^+ \colon B \subseteq A_\alpha^\beta$.
\item $\forall \alpha < \kappa^+ \,\, \forall f_\alpha \text{-small} \, B \subseteq 2^\kappa \colon \alpha > 0 \Rightarrow \big ( \bigcap_{\gamma< \alpha} A_\gamma^0 \big ) \setminus B \neq \emptyset$. This means that for every $\alpha < \kappa^+$ the set $\bigcap_{\gamma< \alpha} A_\gamma^0$ is not $f_\alpha$-small.
\end{itemize}
\end{lemma}

\begin{proof}
We shall construct $A_\alpha^\beta$ by a lexicographic induction on $(\alpha, \beta)\in \kappa^+ \times \kappa^+$: Assume that $(A_\gamma^\beta)_{\gamma < \alpha}^{\beta < \kappa^+}$ have already been defined. Since $\bigcap_{\gamma< \alpha} A_\gamma^0$ is comeager, there exists $f$ such that $\bigcap_{\gamma< \alpha} A_\gamma^0$ is not $f$-small by Lemma \ref{L1}.  W.l.o.g. let $f=f_\alpha$. Choose some $\tau_0 \in \mathcal{S}(f_\alpha)$ such that $Y(\tau_0)$ is comeager and set $A_\alpha^0:= Y(\tau_0)$. Next enumerate $\mathcal{S}(f_\alpha)$ as $(\sigma_\beta)_{\beta < \kappa^+}$ such that $\sigma_0 =\tau_0$. Finally choose $\tau_\beta \in \mathcal{S}(f_\alpha)$ inductively such that $\bigcup_{\gamma < \beta} Y(\tau_\gamma) \cup Y(\sigma_\beta) \subseteq Y(\tau_\beta)$. This is possible since $\mathcal{I}(f_\alpha)$ is $\kappa$-closed. Set $A_\alpha^\beta:=Y(\tau_\beta)$.
\end{proof}

\begin{fact}
If $(f_\alpha)_{\alpha < \kappa^+}$ and $(A_\alpha ^\beta)_{\alpha <\kappa^+} ^{\beta < \kappa^+}$ are as above, then for every $g \in \big (\kappa^+ \big )^{\kappa^+}$ it holds that $\bigcap_{\alpha < \kappa^+} A_\alpha^{g(\alpha)} \in \mathcal{SN}$.
\end{fact}

We are ready to prove the first theorem:

\begin{theorem} \label{T1}
Assume $CH$. Then $\text{cof}(\mathcal{SN})= \mathfrak{d}_{\kappa^+}$.
\end{theorem}

\begin{proof}
First we prove $\text{cof}(\mathcal{SN})\leq \mathfrak{d}_{\kappa^+}$. Let $\mathcal{D}$ be a pointwise dominating family in $\big ({\kappa^+} \big)^{\kappa^+}$ of size $\mathfrak{d}_{\kappa^+}$. This family exists, because there is an eventually dominating family of size $ \mathfrak{d}_{\kappa^+}$ and $\big ({\kappa^+} \big )^{\kappa} = \kappa^+$, since we have $CH$. Let $(f_\alpha)_{\alpha< \kappa^+}$ be a $\kappa$-scale (which exists by $CH$) and $(A_\alpha^\beta)_{\alpha< \kappa^+}^{\beta<\kappa^+}$ the matrix from Lemma \ref{L2}. Define $\mathcal{B}:=\{ X \subseteq 2^\kappa \colon \exists g \in \mathcal{D} \,\, X=\bigcap_{\alpha<\kappa^+} A_\alpha^{g(\alpha)} \}$. $\mathcal{B}$ has size $\leq\mathfrak{d}_{\kappa^+}$. We must show that $\mathcal{B}$ is cofinal in $\mathcal{SN}$. \\
First, it obviously holds that $\mathcal{B} \subseteq \mathcal{SN}$. Now let $Y \in \mathcal{SN}$. Then there exists an $h \in \big ({\kappa^+} \big)^{\kappa^+}$ such that $Y \subseteq \bigcap_{\alpha < \kappa^+} A_\alpha^{h(\alpha)}$. Hence, there exists $g\in \mathcal{D}$ such that $g$ dominates $h$ pointwise. So $Y \subseteq \bigcap_{\alpha < \kappa^+} A_\alpha^{g(\alpha)} \in  \mathcal{B}$.\\
\\
Next, let us show that $\text{cof}(\mathcal{SN})\geq \mathfrak{d}_{\kappa^+}$. Assume that $\text{cof}(\mathcal{SN})<\mathfrak{d}_{\kappa^+}$. So there exists a $\mathcal{C}$ cofinal in $\mathcal{SN}$ of size ${<} \,\mathfrak{d}_{\kappa^+}$. Hence, for every $X \in \mathcal{C}$ there exists $g_X \in \mathcal{D}$ such that $X \subseteq \bigcap_{\alpha < \kappa^+} A_\alpha^{g_X(\alpha)}$. Let us define $\mathcal{D}':=\{g_X \colon X \in \mathcal{C}\} \subseteq \mathcal{D}$. Then $\vert \mathcal{D}' \vert < \mathfrak{d}_{\kappa^+}$. Hence, there exists $h$ such that no $g \in \mathcal{D}'$ dominates it.\\
Inductively we will now construct $h'\in \big ({\kappa^+} \big)^{\kappa^+}$ and $\{x_\gamma \colon \gamma < \kappa^+\}$ such that:
\begin{itemize}
\item $\forall \alpha < \kappa^+ \colon h(\alpha) \leq h'(\alpha)$.
\item $\forall \alpha <\kappa^+ \colon \{x_\gamma \colon \gamma < \alpha\} \subseteq  A_\alpha^{h'(\alpha)}$.
\item $x_\alpha \in \big (\bigcap_{\gamma \leq \alpha} A_\gamma^{h'(\gamma)} \big ) \setminus A_\alpha^{h(\alpha)}$.
\end{itemize}
Assume that $h' \restriction \alpha$ and $\{x_\gamma \colon \gamma < \alpha\}$ have already been defined. Simply by choosing $h'(\alpha)$ large enough we can ensure that $\{x_\gamma \colon \gamma < \alpha\} \subseteq A_\alpha^{h'(\alpha)}$. Also since $\bigcap_{\gamma < \alpha} A_\gamma^0 \subseteq \bigcap_{\gamma < \alpha} A_\gamma^{h'(\gamma)}$, it follows by Lemma \ref{L2} that $\big (\bigcap_{\gamma < \alpha} A_\gamma^{h'(\gamma)} \big ) \setminus  A_\alpha^{h(\alpha)} \neq \emptyset$. Again, if we choose $h'(\alpha)$ large enough, then also $\big (\bigcap_{\gamma \leq \alpha} A_\gamma^{h'(\gamma)} \big )\setminus  A_\alpha^{h(\alpha)} \neq \emptyset$. So choose $h'(\alpha)$ large enough $\geq h(\alpha)$, and pick $x_\alpha \in \big (\bigcap_{\gamma \leq \alpha} A_\gamma^{h'(\gamma)} \big) \setminus  A_\alpha^{h(\alpha)}$.\\
First, $\{x_\gamma \colon \gamma < \kappa^+\} \in \mathcal{SN}$, because $\{x_\gamma \colon \gamma < \kappa^+\} \subseteq \bigcap_{\gamma < \kappa^+} A_\gamma^{h'(\gamma)}$. Finally we show that no $X \in \mathcal{C}$ covers $\{x_\gamma \colon \gamma < \kappa^+\}$. We will show that for every $g_X \in  \mathcal{D}'$ it holds that $\{x_\gamma \colon \gamma < \kappa^+\} \nsubseteq \bigcap_{\alpha < \kappa^+} A_\alpha^{g_X(\alpha)}$. Let $g_X \in  \mathcal{D}'$ be arbitrary. Find $\alpha < \kappa$ such that $g(\alpha) \leq h(\alpha)$. But then $x_\alpha \notin A_\alpha^{g_X(\alpha)}$.
\end{proof}

We can generalize Theorem \ref{T1} as follows:

\begin{theorem} \label{T2}
Assume that $\text{add}(\mathcal{M}_\kappa)=\mathfrak{d}_\kappa$ and there exists a dominating family $\{f_\alpha \in \kappa^\kappa \colon \alpha < \mathfrak{d}_\kappa\}$ such that $\text{add}(\mathcal{I}(f_\alpha))=\text{cof}(\mathcal{I}(f_\alpha)) =\mathfrak{d}_\kappa$. Then $\text{cof}(\mathcal{SN})=\mathfrak{d}_{\mathfrak{d}_\kappa}$.
\end{theorem}

\begin{proof}
First note that also $\mathfrak{b}_\kappa=\mathfrak{d}_\kappa$. Therefore, for every $\beta < \mathfrak{d}_\kappa$, the family $\{f_\alpha \colon \beta \leq \alpha < \mathfrak{d}_\kappa\}$ is also dominating. So w.l.o.g. we can assume that $\{f_\alpha \colon \alpha < \mathfrak{d}_\kappa\}$ is a $\kappa$-scale. Next, construct a matrix $(A_\alpha ^\beta)_{\alpha <\mathfrak{d}_\kappa} ^{\beta < \mathfrak{d}_\kappa}$ similar to Lemma \ref{L2}.\\
Let $\mathcal{D}$ be dominating family in $\mathfrak{d}_\kappa^{\mathfrak{d}_\kappa}$ of size $\mathfrak{d}_{\mathfrak{d}_\kappa}$. Following the proof of Theorem \ref{T1} we define $\mathcal{B}:=\{ X \subseteq 2^\kappa \colon \exists g \in \mathcal{D} \,\, \exists \beta < \mathfrak{d}_\kappa \, \, X=\bigcap_{\alpha \geq \beta} A_\alpha^{g(\alpha)} \}$. Note that $\vert \mathcal{B} \vert\leq \mathfrak{d}_{\mathfrak{d}_\kappa}$ and $\mathcal{B} \subseteq \mathcal{SN}$, since $\{f_\alpha \colon \alpha < \mathfrak{d}_\kappa\}$ is a $\kappa$-scale. We shall show that $\mathcal{B}$ is cofinal in $\mathcal{SN}$:\\
Let $Y \in \mathcal{SN}$, therefore there exists $h \in \mathfrak{d}_\kappa^{\mathfrak{d}_\kappa}$ such that $Y \subseteq \bigcap_{\alpha <\mathfrak{d}_\kappa} A_\alpha^{h(\alpha)}$. But then there is $g \in \mathcal{D}$ and $\beta < \mathfrak{d}_\kappa$ such that $\forall \alpha \geq \beta \colon g(\alpha) \geq h(\alpha)$. Therefore, $Y \subseteq  \bigcap_{\alpha \geq \beta} A_\alpha^{g(\alpha)} \in \mathcal{B}$.\\
To prove $\mathfrak{d}_{\mathfrak{d}_\kappa} \leq \text{cof}(\mathcal{SN})$ we proceed as in the proof of Theorem \ref{T1}.
\end{proof}

\section{A Model}

In this section we want to force $\text{add}(\mathcal{M}_\kappa)=\mathfrak{d}_\kappa$, and there exists a dominating family $\{f_\alpha \colon \alpha < \mathfrak{d}_\kappa\}$ such that $\text{add}(\mathcal{I}(f_\alpha))=\text{cof}(\mathcal{I}(f_\alpha)) =\mathfrak{d}_\kappa$ for every $\alpha < \mathfrak{d}_\kappa$. Then $\text{cof}(\mathcal{SN})= \mathfrak{d}_{\mathfrak{d}_\kappa}$ holds by Theorem \ref{T2}.

\begin{definition}
For $f \in \kappa^\kappa$ strictly increasing define the forcing notion $\mathbb{O}_f$:\\
$p$ is a condition iff $p=(\sigma, \alpha, l, F)$ such that:
\begin{itemize}
\item [P1] $\sigma \in (2^{< \kappa})^{<\kappa}$, $\alpha, l \in \kappa$ and $F \subset \mathcal{S}(f)$ is infinite and of size $< \kappa$
\item [P2] $\vert F \vert \cdot l \geq \text{dom}( \sigma ) \geq l \geq \sup\{M_\tau^{3 \cdot \beta} \colon \tau \in F, \, \beta < \alpha\} \cup \vert F \vert$
\end{itemize}
If $p=(\sigma, \alpha, l, F)$ and $q=(\sigma', \alpha', l', F')$ are conditions in $\mathbb{O}_f$ we define $q \leq_{\mathbb{O}_f} p$, i.e. $q$ is stronger than $p$, iff:
\begin{itemize}
\item [Q1] $\sigma \subseteq \sigma'$, $\alpha \leq \alpha'$, $l \leq l'$ and $F \subseteq F'$
\item [Q2] $\forall \beta < \alpha \,\, \forall i \in \text{dom}(\sigma') \setminus \text{dom}(\sigma) \colon g_{\sigma'}(i) \geq f(i^\beta)$
\item [Q3] $\forall \tau \in F \,\, \forall i \in l' \setminus l \,\, \exists j \in \text{dom}(\sigma') \colon \sigma'(j)=\tau(i)$
\end{itemize}
\end{definition}

The following lemma will be crucial for many density arguments.

\begin{lemma} \label{L3}
Let $p= (\sigma, \alpha, l, F) \in \mathbb{O}_f$. Let $\eta \geq \text{dom}(\sigma)$, $\alpha' \geq \alpha$, $l' \geq l $ and $F' \supseteq F$. Then there exists an extension $q=(\sigma', \alpha', l'',F')$ such that $\text{dom}(\sigma') \geq \eta$ and $l'' \geq l'$.
\end{lemma}

\begin{proof}
Let $\{\tau_k \colon k < \vert F\vert\}$ enumerate $F$ and set $$\tilde{l}:= \sup\{M_\tau^{3 \cdot \beta} \colon \tau \in F', \, \beta < \alpha' \} \cup \max\{ \vert F' \vert, l',\eta\}.$$ Set $l'':=l + \tilde{l}$. Then $l'' \geq l'$. And we define $\sigma'$ as follows:
\begin{itemize}
\item $\sigma' \restriction \text{dom}(\sigma) =\sigma$
\item For $i \in (\text{dom}(\sigma) + \vert F \vert \cdot \tilde{l} ) \setminus \text{dom}(\sigma)$ such that $i= \text{dom}(\sigma) + \vert F \vert \cdot a + b$, where $a < \tilde{l}$ and $b < \vert F \vert$, we set $\sigma'(i):=\tau_b(l+a)$.
\end{itemize}
Then $\text{dom}(\sigma') \geq \eta$. Set $q:=(\sigma', \alpha', l'',F')$. Now we must check that $q \in \mathbb{O}_f$ and $q \leq p$.\\
\\
Let us first check that $q \in \mathbb{O}_f$. The following inequalities hold:
$$ \vert F'\vert \cdot l'' \geq \text{dom}(\sigma') = \text{dom}(\sigma) + \vert F \vert \cdot \tilde{l} \geq l''= l + \tilde{l} \geq \sup\{M_\tau^{3 \cdot \beta} \colon \tau \in F', \, \beta < \alpha'\} \cup \vert F' \vert $$ Therefore, $q \in \mathbb{O}_f$.\\
\\
Now let us check that $q \leq_{\mathbb{O}_f} p$: 
\begin{itemize}
\item [(Q1)] $\sigma \subseteq \sigma'$, $\alpha \leq \alpha'$, $l \leq l''$ and $F \subseteq F'$.
\item [(Q2)] We need to show that $\forall \beta < \alpha \,\, \forall i \in \text{dom}(\sigma') \setminus \text{dom}(\sigma) \colon g_{\sigma'}(i) \geq f(i^\beta)$. Let $\beta < \alpha$ and $i \in \text{dom}(\sigma') \setminus \text{dom}(\sigma)$ be arbitrary such that $i - \text{dom}(\sigma) =\vert F \vert \cdot a +b$. We note that $l+a \geq M_{\tau_b}^{3 \cdot \beta}$ by the definition of $\mathbb{O}_f$, and the following inequalities hold:
$$(l+a)^3 \geq (l+a) \cdot (l+a+1) \geq \vert F \vert \cdot (l+a+1) \geq \vert F \vert \cdot l + \vert F \vert \cdot a +b \geq i$$
Therefore, $g_{\sigma'} (i)= g_{\tau_b}(l+a) \geq f((l+a)^{3 \cdot \beta}) \geq f(i^\beta)$.
\item [(Q3)] Obviously, $\forall \tau \in F \,\, \forall i \in l'' \setminus l \,\, \exists j \in \text{dom}(\sigma') \colon \sigma'(j)=\tau(i)$ by the definition of $\sigma'$ and $l'' - l= \tilde{l}$.
\end{itemize}
\end{proof}

\begin{lemma}
$\mathbb{O}_f$ is ${<}\kappa$-closed.
\end{lemma}

\begin{proof}
Let $(p_k)_{k< \gamma}$ be a decreasing sequence of length $\gamma < \kappa$ such that $p_k=(\sigma_k, \alpha_k, l_k, F_k)$. Define $p:=(\sigma, \alpha, l, F)$, where $\sigma:=\bigcup_{k < \gamma} \sigma_k$, $\alpha:=\bigcup_{k < \gamma} \alpha_k$, $l:=\bigcup_{k < \gamma} l_k$ and $F:=\bigcup_{k < \gamma} F_k$.\\
Let us first check that $p \in \mathbb{O}_f$. Since $\vert F\vert \cdot l \geq \vert F_k\vert \cdot l_k$, the following inequalities hold:
$$\vert F\vert \cdot l \geq \text{dom}(\sigma) \geq l \geq \sup\{M_\tau^{3 \cdot \beta} \colon \tau \in F, \, \beta < \alpha\} \cup \vert F \vert$$
and therefore $p$ is indeed a condition.\\
Next let us check that $p$ is a lower bound of $(p_k)_{k< \gamma}$. Fix $p_k$ and we shall show that $p \leq_{\mathbb{O}_f} p_k$. (Q1) is trivially satisfied. For (Q2) fix $\beta < \alpha_k$ and $i \in \text{dom}(\sigma) \setminus \text{dom}(\sigma_k)$. Choose $k' > k$ such that $i \in \text{dom}(\sigma_{k'})$. But then $g_{\sigma_{k'}}(i) \geq f(i^\beta)$. (Q3) can be shown similarly.
\end{proof}

\begin{definition}
We call a forcing notion $\mathcal{P}$ $\kappa$-linked iff there exists $(X_i)_{i<\kappa}$ such that:
\begin{itemize}
\item $\mathcal{P}= \bigcup_{i< \kappa} X_i$
\item $\forall i < \kappa \colon X_i \,\, \text{is linked}$, i.e. $\forall p_1 ,p_2 \in  X_i, \colon p_1 \,\, \text{and} \,\, p_2 \,\, \text{are compatible}$.
\end{itemize}
\end{definition}

\begin{lemma} \label{L4}
$\mathbb{O}_f$ is $\kappa$-linked \footnote{Note that $\mathbb{O}_f$ is not $\kappa\text{-centered}_{<\kappa}$ (see definition \ref{D1}).}.
\end{lemma}

\begin{proof}
For $\sigma \in (2^{<\kappa})^{<\kappa}$, $\alpha < \kappa$ and $l < \kappa$ define the set $X_{(\sigma, \alpha,l)}:=\{ p \in \mathbb{O}_f \colon \exists F \subset \mathcal{S}(f) \,\,\, p =(\sigma, \alpha,l,F)\}$. We will show that $X_{(\sigma, \alpha,l)}$ is linked. Then $\mathbb{O}_f$ will be $\kappa$-linked, because $\mathbb{O}_f=\bigcup_{\sigma \in (2^{<\kappa})^{<\kappa}} \bigcup_{\alpha < \kappa} \bigcup_{l< \kappa} X_{(\sigma, \alpha,l)}$.\\
Fix $(\sigma, \alpha,l)$ and let $p_1, p_2 \in  X_{(\sigma, \alpha,l)}$. Set $F:=F_{p_1} \cup F_{p_2}$ and note that $\vert F \vert =\max( \vert F_{p_1} \vert, \vert F_{p_2}\vert)$ and therefore, $l \geq \vert F\vert$, which is crucial for the proof. Similar to the proof of Lemma \ref{L3} we can construct a $q=(\sigma', \alpha, l', F)$ which is a lower bound of $p_1$ and $p_2$.
\end{proof}

Assume that $V \vDash \kappa \,\, \text{is inaccessible}$ and let $f \in \kappa^\kappa \cap V$ be strictly increasing. Furthermore, let $G$ be a $( \mathbb{O}_f , V)$-generic filter. Define $\tau_G:=\bigcup\{\sigma \in (2^{<\kappa})^{<\kappa}\colon \exists p \in G \, \, p=(\sigma, \alpha_p, l_p, F_p)\}$.  Then the following lemma is an easy observation:

\begin{lemma}
The following holds in $V^{{\mathbb{O}_f}}$:

\begin{enumerate}
\item $\tau_G \in (2^{<\kappa})^\kappa$
\item $g_{\tau_G} \gg f$, so $\tau_G \in \mathcal{S}(f)$
\item $ \forall \tau \in \mathcal{S}(f) \cap V \colon Y(\tau) \subset Y(\tau_G)$
\end{enumerate}
So $\tau_G$ codes a $f$-small set which covers all ground model $f$-small sets.
\end{lemma}

\begin{proof}
\begin{enumerate}
\item By Lemma \ref{L3} the set $\{p \in \mathbb{O}_f \colon \text{dom}(\sigma_p)\geq \eta\}$ is dense for all $\eta < \kappa$. Therefore, $\tau_G \in (2^{<\kappa})^\kappa$.
\item Let $\alpha <\kappa$ be arbitrary. By a density argument there exists $p \in G$ such that $\alpha_p \geq \alpha+1$. But then $g_{\tau_G}(i) \geq f(i^\alpha)$ for all $i \geq \text{dom}(\sigma_p)$.
\item Let $\tau \in \mathcal{S}(f) \cap V $ be arbitrary and fix $x \in Y(\tau)$, so the set $\{i <\kappa \colon x \in [\tau(i)]\}$ has size $\kappa$. By a density argument there exists $p,q \in G$ such that $\tau \in F_p$, $l_q$ is arbitrarily large and $q\leq_{\mathbb{O}_f} p$. Hence the set $\{i \geq \text{dom}(\sigma_p) \colon x \in [\tau_G(i)]\}$ will also be of size $\kappa$.
\end{enumerate}
\end{proof}

\begin{definition}
Let $\lambda > \kappa$ be a regular cardinal. Define $B^\lambda$ to be a bijection between $\lambda$ and $\lambda \times \lambda$ such that:
\begin{itemize}
\item If $B^\lambda(\alpha)=(\beta,\gamma)$ then $\beta \leq \alpha$
\item If $B^\lambda(\alpha)=(\beta, \gamma)$, $B^\lambda(\alpha')=(\beta, \gamma')$ and $\alpha < \alpha'$ then $\gamma < \gamma'$
\end{itemize}
Furthermore, define $B_0^\lambda(\alpha)$ and $B_1^\lambda(\alpha)$ to be the projection of $B^\lambda(\alpha)$ onto the first and second coordinate, respectively.
\end{definition}

Now we are ready to define the iteration:

\begin{definition}
Let $(\mathbb{P}_\epsilon, \dot{\mathbb{Q}}_\beta \colon \epsilon \leq \lambda, \beta < \lambda)$ be a ${<}\kappa$-support iteration such that $\forall \epsilon < \lambda  \colon \Vdash_\epsilon \dot{\mathbb{Q}_\epsilon}= \dot{\mathbb{H}_\kappa} * \dot{\mathbb{O}}(\dot{d}_{B_0^\lambda(\epsilon)})$, where $\mathbb{H}_\kappa$ denotes $\kappa$-Hechler forcing and $\dot{d}_{B_0^\lambda(\epsilon)}$ is the generic $\kappa$-Hechler added by the first half of $\dot{\mathbb{Q}}_{B_0^\lambda(\epsilon)}$.
\end{definition}

Note that $\mathbb{H}_\kappa$ is $\kappa$-linked. Even more, it is $\kappa$-$\text{centered}_{<\kappa}$ :

\begin{definition} \label{D1}
We call a forcing notion $\mathcal{P}$ $\kappa$-$\text{centered}_{<\kappa}$ iff there exists $(X_i)_{i<\kappa}$ such that:
\begin{itemize}
\item $\mathcal{P}= \bigcup_{i< \kappa} X_i$
\item $\forall i < \kappa \colon X_i \,\, \text{is} \,\, \text{centered}_{<\kappa}$, i.e. $\forall A \subset  X_i, \colon \vert A \vert < \kappa \Rightarrow \exists q \in \mathcal{P} \,\, q \,\, \text{is a lower bound of} \,\, A$.
\end{itemize}
\end{definition}

The following lemma should be a straightforward consequence of Lemma \ref{L4}:

\begin{lemma}
$\mathbb{P}_\lambda$ satisfies the $\kappa^+$-c.c. Furthermore, if $\vert 2^\kappa\vert <\lambda$ and $\lambda^\kappa= \lambda$ or $\vert2^\kappa\vert \geq \lambda$, then there exists a dense set $D \subseteq \mathbb{P}_\lambda$ of size $\max(\vert2^\kappa \vert,\lambda)$.
\end{lemma}

\begin{proof}
The set $$D:=\{p \in \mathbb{P}_\lambda \colon \forall \epsilon < \lambda \, \, \exists \rho \in \kappa^{<\kappa} \,\, \exists \sigma \in (2^{<\kappa})^{<\kappa} \,\,  \exists \alpha < \kappa \,\, \exists l < \kappa \,\, \exists \dot{f} \,\, \exists (\dot{g}_i)_{i<\gamma} $$ $$p \restriction \epsilon \Vdash_\epsilon  \dot{p(\epsilon)}=((\rho, \dot{f}),(\sigma, \alpha, l, (\dot{g}_i)_{i<\gamma}))\}$$
is dense in $\mathbb{P}_\lambda$. Let $A\subset D$ be of size $\kappa^+$ and use a $\Delta$-system argument to find $B \subset A$ of size $\kappa^+$ such that $B$ is linked. Hence $\mathbb{P}_\lambda$ satisfies the $\kappa^+$-c.c.\\
Show by induction over $\epsilon < \gamma$ that $\vert D \cap \mathbb{P}_\epsilon\vert \leq \max(\vert 2^\kappa \vert,\lambda)$ by using the $\kappa^+$-c.c. of $\mathbb{P}_\epsilon$ and the fact that every $\mathbb{P}_\epsilon$-name for an element of $\kappa^\kappa$ is completely determined by a family of maximal antichains of size $\kappa$.
\end{proof}

We are ready to state and prove the following theorem:

\begin{theorem}
Let $\lambda > \kappa$ be regular. Let $V \vDash \vert 2^\kappa\vert \geq \lambda \,\, \text{or} \,\, \vert2^\kappa\vert < \lambda \land \lambda^\kappa=\lambda$. Then in $V^{\mathbb{P}_\lambda}$ the following holds: $\text{add}(\mathcal{M}_\kappa)=\mathfrak{d}_\kappa= \lambda$ and there exists a dominating family $\{f_\alpha \in \kappa^\kappa \colon \alpha < \lambda\}$ such that $\text{add}(\mathcal{I}(f_\alpha))=\text{cof}(\mathcal{I}(f_\alpha)) =\lambda$. Therefore $\text{cof}(\mathcal{SN})=\mathfrak{d}_{\mathfrak{d}_\kappa}= \mathfrak{d}_\lambda$.\\
Furthermore, $\vert 2^\kappa\vert =\max(\vert 2^\kappa \cap V \vert, \lambda)$, $\mathfrak{d}_\lambda=\mathfrak{d}_\lambda^V$ and $\text{add}(\mathcal{SN})= \text{cov}(\mathcal{SN})=\text{non}(\mathcal{SN})=\lambda$.
\end{theorem}

\begin{proof}
Using the $\kappa^+$-c.c., the following should be straightforward: The $\kappa$-Hechler reals $(\dot{d}_\epsilon)_{\epsilon < \lambda}$ form a $\kappa$-scale, so $\mathfrak{d}_\kappa=\mathfrak{b}_\kappa=\lambda$. A ${<}\kappa$-support iteration adds cofinally many Cohen reals, therefore also $\text{cov}(\mathcal{M}_\kappa)=\lambda$, and since $\text{add}(\mathcal{M}_\kappa)= \min( \mathfrak{b}_\kappa, \text{cov}(\mathcal{M}_\kappa)$ by \cite{Baumhauer}, it follows that $\text{add}(\mathcal{M}_\kappa)=\mathfrak{d}_\kappa= \lambda$. The family $(\dot{\tau}_{\epsilon'})_{ \epsilon' \in B^{\lambda \, ^{-1}} (\{\epsilon\} \times \lambda)}$ witnesses $\text{add}(\mathcal{I}(\dot{d}_\epsilon))=\text{cof}(\mathcal{I}(\dot{d}_\epsilon)) =\lambda$, since $\mathcal{I}(\dot{d}_\epsilon))$ is a $\kappa$-Borel ideal. Hence $\text{cof}(\mathcal{SN})=\mathfrak{d}_{\mathfrak{d}_\kappa}$ by Theorem \ref{T2}.\\
Since $\mathbb{P}_\lambda$ has a dense subset $D$ of size $\max(\vert 2^\kappa \cap V \vert, \lambda)$, it satisfies the $\kappa^+$-c.c. and $\lambda^\kappa=\lambda$, it follows that $\vert 2^\kappa\vert \leq \vert\max(\vert 2^\kappa \cap V \vert, \lambda)$. Again using the $\kappa^+$-c.c., it follows that $\mathbb{P}_\lambda$ is $\lambda^\lambda$-bounding, hence $\mathfrak{d}_\lambda=\mathfrak{d}_\lambda^V$. Using $\text{add}(\mathcal{I}(\dot{d}_\epsilon))=\lambda$ for all $\epsilon< \lambda$, we can deduce $\lambda \leq \text{add}(\mathcal{SN})$. For $\text{cov}(\mathcal{SN}) \leq \lambda$ we note that $2^\kappa \cap V^{\mathbb{P}_\epsilon} \in \mathcal{SN}$ for all $\epsilon<\lambda$ as being witnessed by $(\dot{\tau}_{\epsilon'})_{\epsilon'> \epsilon}$. For $\text{non}(\mathcal{SN}) \leq \lambda$ we pick for every $\dot{\tau}_\epsilon$ possibly not distinct $x_\epsilon$'s such that $x_\epsilon \in 2^\kappa \setminus Y(\dot{\tau}_\epsilon)$ and set $X:=\{x_\epsilon \colon \epsilon< \lambda\}$. It follows that $\vert X \vert \leq \lambda$ and $X \notin \mathcal{SN}$, since no $\dot{\tau}_\epsilon$ can cover $X$.
\end{proof}

\begin{corollary}
$\mathfrak{c}_\kappa < \text{cof}(\mathcal{SN})$, $\mathfrak{c}_\kappa = \text{cof}(\mathcal{SN})$ and $\mathfrak{c}_\kappa > \text{cof}(\mathcal{SN})$ are all consistent relative to ZFC.
\end{corollary}

\begin{proof}
$\mathfrak{c}_\kappa < \text{cof}(\mathcal{SN})$ holds under $CH$. For $\mathfrak{c}_\kappa = \text{cof}(\mathcal{SN})$ assume that $V \vDash \vert 2^\kappa \vert = \kappa^{++} \land \mathfrak{d}_{\kappa^+}=\kappa^{++}$ (e.g. by forcing over $GCH$ with a $\kappa^{++}$-product of $\kappa$-Cohen forcing) and force with $\mathbb{P}_{\kappa^+}$. For $\mathfrak{c}_\kappa > \text{cof}(\mathcal{SN})$ assume $V \vDash \vert 2^\kappa \vert = \kappa^{+++} \land \mathfrak{d}_{\kappa^+}=\kappa^{++}$ (e.g. by forcing over $GCH$ with a $\kappa^{+++}$-product of $\kappa$-Cohen forcing) and again force with  $\mathbb{P}_{\kappa^+}$.
\end{proof}

\section{A model for $\text{cov}(\mathcal{SN}) < \text{add}(\mathcal{M}_\kappa)$}

In this section we first generalize the Galvin–Mycielski–Solovay theorem to $\kappa$ inaccessible. Then we shall assume that $\kappa$ is strongly unfoldable, and want to construct a model for $\text{cov}(\mathcal{SN}) < \text{add}(\mathcal{M}_\kappa)$. Indeed this will hold in the $\kappa$-Hechler model.

\begin{definition}
Let $X \subset 2^\kappa$. We say that $X$ is meager-shiftable iff $\forall D \subset 2^\kappa$, $D= \bigcap_{i<\kappa} D_i$ for some $(D_i)_{i<\kappa}$ open dense, there exists $y \in 2^\kappa$ such that $X + y \subset D$.
\end{definition}

The following result is well-known in the classical case:

\begin{theorem} \label{T3}
{(GMS)} $X \in \mathcal{SN}$ iff $X$ is meager-shiftable.
\end{theorem}

\begin{proof}
First we shall show that meager-shiftable implies strong measure zero: Let $X \subset 2^\kappa$ be meager-shiftable and let $f \in \kappa^\kappa$. Choose $(s_i)_{i < \kappa}$ such that $s_i \in 2^{f(i)}$ and $D:= \bigcup_{i < \kappa} [s_i]$ is open dense. Since $X$ is meager-shiftable, there exists $y \in 2^\kappa$ such that $X+y \subset D$. Define $t_i:= s_i + y \restriction f(i)$. But then $X \subset \bigcup_{i< \kappa} [t_i]$. Hence $X$ is strong measure zero.\\
\\
We shall now show that strong measure zero implies meager-shiftable: Let $X \subset 2^\kappa$ be strong measure zero and let $D=\bigcap_{i<\kappa} D_i$ be an intersection of arbitrary dense open sets. W.l.o.g. let the $D_i$'s be decreasing. Now construct a normal sequence $(c_i)_{i<\kappa}$,  $c_i \in \kappa$, such that for every $i < \kappa$ and every $s \in 2^{c_i}$ there exists $t \in 2^{c_{i+1}}$ such that $[t] \subset D_i$. Define $f(i):=c_{i+1}$ and find $(s_i)_{i<\kappa}$ such that $s_i \in 2^{f(i)}$ and $X \subset \bigcap_{j< \kappa} \bigcup_{i\geq j} [s_i]$. We shall now inductively construct $y \in 2^\kappa$ such that $\forall i < \kappa \colon [s_i + y \restriction c_{i+1}] \subset D_i$:

\begin{itemize}
\item $i=0$: Choose $t_0 \in 2^{c_1}$ such that $[t_0] \subset D_0$ and set $y \restriction c_1:= s_0 + t_0$. Hence, $s_0 +y \restriction c_1 =t_0$ and so $[s_0 + y \restriction c_1] \subset D_0$.
\item $i=i'+1$: Assume that $y \restriction c_{i'+1}$ has already be constructed. Find $t_i \,  \triangleright \, s_i \restriction c_i + y \restriction c_i$, $t_i \in 2^{c_{i+1}}$, such that $[t_i] \subset D_i$. Set $y \restriction c_{i+1}:= s_i +t_i$. Then $s_i+ y \restriction c_{i+1} = t_i$ and $y \restriction c_{i+1} \triangleright y \restriction c_i$.
\item $i$ is a limit: Set $y \restriction c_i:= \bigcup_{j< i} y \restriction c_j$ and proceed as in the successor step to construct $y \restriction c_{i+1}$.
\end{itemize}
Now we will show that $ \forall i < \kappa \colon X+y \subset D_i$. To this end let $i < \kappa$ and $x \in X$ be arbitrary. We can now find $i' \geq i$ such that $s_{i'} \triangleleft x$. It follows that $x+y \in [s_{i'}+ y \restriction c_{i'+1}] \subset D_{i'} \subset D_i$. This finishes the proof.
\end{proof}

\begin{definition} \label{D3}
We call a cardinal $\kappa$ strongly unfoldable iff $\kappa$ is inaccessible and for every cardinal $\theta$ and every $A \subset \kappa$ there exists a transitive model $M$, such that $A \in M$ and $M \vDash \text{ZFC}$, and an elementary embedding $j \colon M \rightarrow N$ with critical point $\kappa$, such that $j(\kappa) \geq \theta$ and $V_\theta \subset N$.
\end{definition}

Before we can construct the model, we will need some definitions: 

\begin{definition}
We say that a forcing notion $\mathcal{P}$ has precaliber $\kappa^+$ iff for every $X \in [\mathcal{P}]^{\kappa^+}$ there exists $Y \in [X]^{\kappa^+}$ such that $Y$ is $\text{centered}_{<\kappa}$ (see Definition \ref{D1}).
\end{definition}

While the previous definition is about forcing in general, the next definition is concerned with $\text{cov}(\mathcal{SN})$:

\begin{definition} \label{D2}
Let $\{\mathcal{D}_\alpha \colon \alpha < \kappa^+\}$ be a sequence of families of open subsets of $2^\kappa$. We say that the family $\{\mathcal{D}_\alpha \colon \alpha < \kappa^+\}$ is good iff $ \forall X \in [\kappa^+]^{\kappa^+} \colon \bigcup_{\alpha \in X} \bigcap \mathcal{D}_\alpha = 2^\kappa$.
\end{definition}

The motivation behind Definition \ref{D2} is that for each $\alpha < \kappa^+$ the set $\bigcap \mathcal{D}_\alpha$ could be a strong measure zero set. Then a good family corresponds to a family of strong measure zero sets of size $\kappa^+$ such that every subfamily of size $\kappa^+$ covers $2^\kappa$.

\begin{lemma}
Suppose that $\{\mathcal{D}_\alpha \colon \alpha < \kappa^+\}$ is a good family in $V$, $V \vDash \kappa \,\, \text{is weakly compact}$ and $\mathcal{P}$ is a ${<}\kappa$-closed forcing notion, which has precaliber $\kappa^+$. Then $\{\mathcal{D}_\alpha \colon \alpha < \kappa^+\}$ is also good in $V^{\mathcal{P}}$.
\end{lemma}

\begin{proof}
Towards a contradiction assume that there are names $\dot{x}$, $\dot{X}$ and a condition $p$ such that $p \Vdash_{\mathcal{P}} \dot{x} \in 2^\kappa \land \dot{X} \in [\kappa^+]^{\kappa^+} \land \dot{x} \notin \bigcup_{\alpha \in \dot{X}} \bigcap \mathcal{D}_\alpha$.  Working in $V$, for every $\alpha < \kappa^+$ find $\eta_\alpha > \alpha$, $D_\alpha \in \mathcal{D}_{\eta_\alpha}$ and conditions $p_\alpha \leq p$ such that $p_\alpha \Vdash_{\mathcal{P}} \eta_\alpha \in \dot{X} \land \dot{x} \notin D_\alpha$. Because $\mathcal{P}$ has precaliber $\kappa^+$ we can find $A \in [\kappa^+]^{\kappa^+}$ such that $\{p_\alpha \colon \alpha \in A\}$ is $\text{centered}_{<\kappa}$. Let $Z \subset A$ be of size $<\kappa$ and let $q_Z$ be a lower bound of $\{p_\alpha \colon \alpha \in Z\}$. Then $q_Z \Vdash_{\mathcal{P}} \bigcup_{\alpha \in Z} D_\alpha \neq 2^\kappa$. By $\Pi_1^1$-absoluteness \footnote{This also guarantees that if $\mathbf{B}_1, \mathbf{B}_2 \in V$ are $\kappa$-Borel codes, then $V \vDash \mathbf{B}_1 =\mathbf{B}_2$ iff $V^{\mathcal{P}} \vDash \mathbf{B}_1 =\mathbf{B}_2$.} for ${<}\kappa$-closed forcing extensions (see \cite{Friedman}) $\bigcup_{\alpha \in Z} D_\alpha \neq 2^\kappa$ must hold in $V$. Since $\kappa$ is weakly compact it follows that $\bigcup_{\alpha \in A} D_\alpha \neq 2^\kappa$. Hence $\bigcup_{\beta \in \{ \eta_\alpha \colon \alpha \in A\}} \bigcap \mathcal{D}_\beta \neq 2^\kappa$. However, this is a contradiction to $\{\mathcal{D}_\alpha \colon \alpha < \kappa^+\}$ being a good family in $V$.
\end{proof}

We are now ready to prove the main theorem of this section:

\begin{theorem}
Let $V$ satisfy $\vert 2^\kappa\vert =\kappa^+$ and the strong unfoldability of $\kappa$ is indestructible by ${<}\kappa$-closed, $\kappa^+$-c.c. forcing notions (see \cite{Johnstone}). Define $\mathbb{P}$ to be a ${<}\kappa$-support iteration of $\kappa$-Hechler forcing of length $\kappa^{++}$. Then $V^{\mathbb{P}} \vDash \text{cov}(\mathcal{SN})= \kappa^+ <\text{add}(\mathcal{M}_\kappa)=\kappa^{++}$.
\end{theorem}

\begin{proof}
Since $\text{add}(\mathcal{M}_\kappa)= \min( \mathfrak{b}_\kappa, \text{cov}(\mathcal{M}_\kappa)$, it follows easily that $\text{add}(\mathcal{M}_\kappa)=\kappa^{++}$ in $V^{\mathbb{P}}$.
Hence we will only show $\text{cov}(\mathcal{SN})= \kappa^+$:\\
\\
Working in $V^{\mathbb{P}}$ we define for $\alpha < \kappa^{++}$  the set $$\mathcal{D}_\alpha:=\{D \colon D \, \text{is dense open with} \,\, \kappa\text{-Borel code in} \, V^{\mathbb{P}_{\alpha+1}} \land 2^\kappa \cap V^{\mathbb{P}_\alpha} \subseteq D\}.$$
For $X\subseteq \kappa^{++}$ define $E(X):= \bigcap_{\alpha \in X} \bigcap \mathcal{D}_\alpha$. If $X$ is cofinal in $\kappa^{++}$, then $E(X)$ is smz: By GMS (see Theorem \ref{T3}) it is enough to show that $E(X)$ is meager-shiftable: Let $D=\bigcap_{i<\kappa} D_i$ be an intersection of arbitrary dense open sets and find $\alpha' \in X$ such that $D$ is coded in $V^{\mathbb{P}_{\alpha'}}$. But then $2^\kappa \cap V^{\mathbb{P}_{\alpha'}} \subseteq D + c_{\alpha'}$, where $c_{\alpha'}$ is some Cohen real over $V^{\mathbb{P}_{\alpha'}}$ added by the next iterand of Hechler forcing. Hence $D + c_{\alpha'} \in \mathcal{D}_{\alpha'}$, and therefore $E(X) \subseteq \bigcap \mathcal{D}_{\alpha'} \subseteq D + c_{\alpha'}$.\\
\\
We claim that $\forall x \in 2^\kappa \colon \vert \{\alpha < \kappa^{++} \colon x \notin \bigcap \mathcal{D}_\alpha\} \vert < \kappa^+$ holds in $V^{\mathbb{P}}$. Towards a contradiction assume that $\alpha^* < \kappa^{++}$ is the minimal ordinal such that there exists $X \in [\alpha^*]^{\kappa^+}$ with $\bigcup_{\alpha \in X} \bigcap \mathcal{D}_\alpha \neq 2^\kappa$. This observation means that the family $\{\mathcal{D}_\alpha \colon \alpha \in \alpha^*\}$ (note that $\vert \alpha^* \vert= \kappa^+$) is not good in $V^{\mathbb{P}}$. Since the family $\{\mathcal{D}_\alpha \colon \alpha \in \alpha^*\}$ is in $V^{\mathbb{P}_{\alpha^*}}$, $\kappa$ is still weakly compact in   $V^{\mathbb{P}_{\alpha^*}}$, and the quotient forcing $\mathbb{P}/G_{\alpha^*}$ is ${<}\kappa$-closed and has precaliber $\kappa^+$, it follows by the previous lemma that $\{\mathcal{D}_\alpha \colon \alpha \in \alpha^*\}$ is also not good in $V^{\mathbb{P}_{\alpha^*}}$. By the minimality of $\alpha^*$ it follows that $X$ is cofinal in $\alpha^*$ and $\text{otp}(X)= \kappa^+$, hence $\text{cf}(\alpha^*)=\kappa^+$. Now working in $V^{\mathbb{P}_{\alpha^*}}$, for any $Y \in [\alpha^*]^{\kappa^+}\cap V^{\mathbb{P}_{\alpha^*}}$ if $Y$ is not cofinal in $\alpha^*$, then $2^\kappa \cap V^{\mathbb{P}_{\alpha^*}} \subseteq \bigcup_{\alpha \in Y} \bigcap \mathcal{D}_\alpha$ must hold because of the minimality of $\alpha^*$. If $Y$ is cofinal in $\alpha^*$, then it holds that $2^\kappa \cap V^{\mathbb{P}_{\alpha^*}} = \bigcup_{\alpha \in Y} 2^\kappa \cap V^{\mathbb{P}_{\alpha}} \subseteq \bigcup_{\alpha \in Y} \bigcap \mathcal{D}_\alpha$. But this is a contradiction to $\{\mathcal{D}_\alpha \colon \alpha \in \alpha^*\}$ not being good in $V^{\mathbb{P}_{\alpha^*}}$.\\
\\
Now let $\{X_\xi \colon \xi < \kappa^+\}$ be a partition of $\kappa^{++}$ into cofinal subsets. It follows from the above claim that $\bigcup_{\xi < \kappa^+} E(X_\xi)= 2^\kappa$, because for every $x \in 2^\kappa$ there must exist $\xi < \kappa^+$ such that for every $\alpha \in X_\xi$ it holds that $x \in \bigcap \mathcal{D}_\alpha$. Hence $\text{cov}(\mathcal{SN})= \kappa^+$.
\end{proof}

\newpage
\bibliography{refs} 
\bibliographystyle{alpha}

\end{document}